\documentclass[12pt]{article}
\usepackage{amsmath, amsthm, amsfonts, amscd, amssymb, graphicx}
\usepackage[letterpaper, margin={0.8in}]{geometry}
\usepackage{enumerate}
\usepackage{setspace,color}
\newtheorem{theorem}{\bf Theorem}[section]

\newtheorem{remark}{\bf Remark}[section]
\newtheorem{lemma}{\bf Lemma}[section]
\newcommand{\eps}{\varepsilon}

\newcommand{\cg}{\mathcal{C}_g}
\newcommand{\tr}{\mathrm{tr}\,}
\newcommand{\dvr}{\mathrm{div}}

\newcommand{\e}{\varepsilon}
\newcommand{\R}{\mathbb R}
\newcommand{\M}{\mathcal M}
\newcommand{\rr}{{\bf r}}
\newcommand{\T}{{\bf T}}
\newcommand{\N}{{\bf N}}
\newcommand{\p}{{\bf p}}

\newcommand{\beq}{\begin{equation}}
\newcommand{\eeq}{\end{equation}}
\newcommand{\abs}[1]{\left\vert{#1}\right\vert}

\title{Dimension reduction for the Landau-de Gennes model on curved nematic thin films}

\author{Dmitry Golovaty\\
\small Department of Mathematics\\
\small The University of Akron\\
\small Akron, OH 4325-4002, USA\\
\small \texttt{dmitry@uakron.edu}\\
\and
Jos\'e Alberto Montero\\
\small Departamento de Matem\'aticas\\
\small Facultad de Matem\'aticas\\
\small Pontificia Universidad Cat\'olica de Chile\\
\small Vicu\~na Mackenna 4860\\
\small San Joaqu\'in, Santiago, Chile\\
\small \texttt{amontero@mat.puc.cl}
\and
Peter Sternberg\\
\small Department of Mathematics\\
\small Indiana University\\
\small Bloomington, IN 47405\\
\small \texttt{sternber@indiana.edu}
}
\date{}

\begin{document}
\maketitle

\begin{abstract}
We use the method of $\Gamma$-convergence to study the behavior of the Landau-de Gennes model for a nematic liquid crystalline film attached to a general fixed surface in the limit of vanishing thickness. This paper generalizes the approach in \cite{GMS} where we considered a similar problem for a planar surface. Since the anchoring energy dominates when the thickness of the film is small, it is essential to understand its influence on the structure of the minimizers of the limiting energy. In particular, the anchoring energy dictates the class of admissible competitors and the structure of the limiting problem. We assume general weak anchoring conditions on the top and the bottom surfaces of the film and strong Dirichlet boundary conditions on the lateral boundary of the film when the surface is not closed. We establish a general convergence result to an energy defined on the surface that involves a somewhat surprising remnant of the normal component of the tensor gradient. Then we exhibit one effect of curvature through an analysis of the behavior of minimizers to the limiting problem when the substrate is a frustum.
\end{abstract}

\section{Introduction}
In this paper we expand our analysis of thin nematic liquid crystalline films, initiated in \cite{GMS} for planar films, to include the setting of general smooth surfaces. The focus of the present work is on rigorous dimensional reduction of the Landau-de Gennes $Q$-tensor model to its surface analog, in particular, to justify asymptotic arguments in \cite{Napoli_Vergori} (see also \cite{virga_talk}).  The Landau-de Gennes theory is based on the $Q$-tensor order parameter field that is related to the second moment of the local orientational probability distribution. The relevant variational model involves minimization of an energy functional consisting of elastic, bulk and weak anchoring surface contributions. The significance of weak anchoring energy terms within both $Q$-tensor and director theories has been highlighted in numerous recent contributions, including for example, \cite{apala_zarnescu_01,canevari13,Lamy14,GM,segatti14,ball2010nematic}.

Having already established in \cite{GMS} the dimension reduction for a planar film, we now wish to explore the possible influence of curvature on the limiting energy in the thin film limit. To achieve this goal we use the theory of $\Gamma$-convergence that has proved successful in tackling problems of dimension reduction in other settings, such as elasticity \cite{anzellotti1994dimension} and Ginzburg-Landau theory \cite{contreras2010gamma}.  

In Section \ref{s:model}  we define the full three-dimensional energy, perform non-dimensionalization, and review some elementary facts from calculus on surfaces. In Section \ref{s:conv} we prove $\Gamma$-convergence to a limiting energy $F_0$, cf. Theorem \ref{t1}. One feature of the $\Gamma$-limit derived in Section \ref{s:conv} is that it includes within its definition a minimum of a certain scalar function defined over the set of traceless symmetric tensors. This minimization arises as a sort of remnant of the normal component of the $Q$-tensor gradient. In Section \ref{s:fezero}, we carry out this minimization thereby obtaining an explicit formula for the $\Gamma$-limit. The formula demonstrates that the limiting energy density contains a number of previously unreported elastic ``strange" terms coupling the surface gradient of the $Q$-tensor to the normal to the surface. Next, as an example, in Section \ref{s:revolve} we compute the expression for the limiting energy in the geometry of a surface of revolution. Specializing further in Section \ref{s:cone}, we analyze the case of a frustum. We discover a dichotomy between the behavior of minimizers for broad and for narrow cones when the nematic coherence length is small. When the angle of the frustum is small, the director field tends to follow the generators of the frustum. However, the director deviates from such a path significantly when the angle broadens and eventually approaches a constant state. As the result, we observe that the degree of the director along the boundary components depends on the angle of the frustum.  

\section{Statement of the problem}
\label{s:model}
\subsection{The $Q$-tensor}
\label{s:qtens}
In the three-dimensional setting, one describes a nematic liquid crystal by a $2$-tensor $Q$ which takes the form of a $3\times  3$ symmetric, traceless matrix. Here $Q(x)$ models the second moment of the orientational distribution of the rod-like molecules near $x$. The tensor $Q$  has three real eigenvalues satisfying $\lambda_1+\lambda_2+\lambda_3=0$ and a mutually orthonormal eigenframe $\left\{\mathbf{l},\mathbf{m},\mathbf{n}\right\}$. We refer the reader to \cite{Mottram_Newton} for more details but below we summarize the key elements of this theory that we will utilize.

Suppose that $\lambda_1=\lambda_2=-\lambda_3/2.$ Then the liquid crystal is in a {\em uniaxial nematic} state and \begin{equation}Q=-\frac{\lambda_3}{2}\mathbf{l}\otimes\mathbf{l}-\frac{\lambda_3}{2}\mathbf{m}\otimes\mathbf{m}+
\lambda_3\mathbf{n}\otimes\mathbf{n}=S\left(\mathbf{n}\otimes\mathbf{n}-\frac{1}{3}\mathbf{I}\right),\label{uniaxial}
 \end{equation}
 where $S:=\frac{3\lambda_3}{2}$ is the uniaxial nematic order parameter and $\mathbf{n}\in\mathbb{S}^2$ is the nematic director and
\[\mathbf{l}\otimes\mathbf{l}+\mathbf{m}\otimes\mathbf{m}+\mathbf{n}\otimes\mathbf{n}=\mathbf{I}.\]
If there are no repeated eigenvalues, the liquid crystal is in a {\em biaxial nematic} state and
\begin{multline}
Q=\lambda_1\mathbf{l}\otimes\mathbf{l}+\lambda_3\mathbf{n}\otimes\mathbf{n}-\left(\lambda_1+\lambda_3\right)\left(\mathbf{I}-\mathbf{l}\otimes\mathbf{l}-\mathbf{n}\otimes\mathbf{n}\right)\\=S_1\left(\mathbf{l}\otimes\mathbf{l}-\frac{1}{3}\mathbf{I}\right)+S_2\left(\mathbf{n}\otimes\mathbf{n}-\frac{1}{3}\mathbf{I}\right),
\label{biaxial}\end{multline}
where $S_1:=2\lambda_1+\lambda_3$ and $S_2=\lambda_1+2\lambda_3$ are biaxial order parameters.
Note that uniaxiality can also be described in terms of $S_1$ and $S_2$, that is one of the following three cases occurs: $S_1=0$ but $S_2\not=0$, $S_2=0$ but $S_1\not =0$ or $S_1=S_2\not=0.$ When $S_1=S_2=0$ so that ${\bf Q}=0$ the nematic liquid crystal is said to be in an isotropic state associated, for instance, with a high
temperature regime.

From the modeling perspective it turns out that the eigenvalues of $Q$ must satisfy the constraints \cite{ball2010nematic,sonnet2012dissipative}:
\begin{equation}
\label{eq:bnds}
\lambda_i\in[-1/3,2/3],\ \mathrm{for}\ i=1,2,3.
\end{equation}

\subsection{Geometry of the Domain}
We will use $X$ to denote a point in $\R^3$. 
\begin{figure}[htb]
\centering
\includegraphics[height=2.5in]{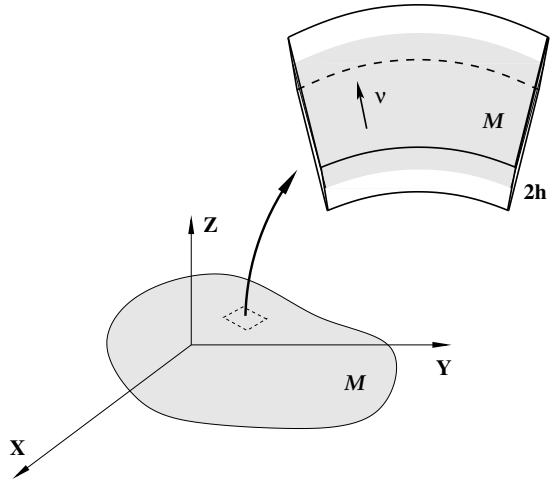}
\caption{Geometry of the problem.}
  \label{fig:1}
\end{figure}
We let $\mathcal{M}$ denote a bounded, two-dimensional, $C^2$ orientable 
manifold embedded in $\R^{3}$, either closed or with smooth boundary, and we let $x$ denote a point on
$\mathcal{M}$. Fixing an orientation, we write $\nu(x)$ for the unit normal and due to the $C^2$ smoothness of $\mathcal{M}$ we have that the mapping 
$x\mapsto\nu(x)$ is $C^1$ on the compact set $\mathcal{M}$. It follows from the inverse function theorem that for some sufficiently small positive number $h_0$, the map $(x,t)\mapsto x+h_0t\nu(x)$ is one-to-one on $\mathcal{M}\times (-1,1).$

With this observation in hand, for $0<h<h_0$ we shall assume the nematic film occupies a thin neighborhood of $\mathcal{M}$ given
by
\begin{displaymath}
\Omega _{h } := \{ X\in\R^3:\;X=x +h t \nu
(x)\;\mbox{for}\; x \in \mathcal{M} ,\; t \in (-1,1)  \}
\end{displaymath}
and we can unambiguously express each point $X\in \Omega_{h}$ 
in the form
\begin{equation}
\label{eq:contras}
X=x+h t\nu(x)
\end{equation}
for some unique pair $x\in\mathcal{M}$ and $t\in(-1,1)$.

 We will also set
\begin{equation}
\label{eq:surfj}
\mathcal{M}_{\pm h}:=\left\{x\pm h\nu(x):x\in\mathcal{M}\right\}.
\end{equation}

\subsection{Landau-de Gennes model}
We assume that the bulk elastic energy density of a nematic liquid crystal is given by
\begin{multline}
\label{elastic}
f_e(\nabla Q):=\frac{L_1}{2}{|\nabla Q|}^2+\frac{L_2}{2}Q_{ij,j}Q_{ik,k}+\frac{L_3}{2}Q_{ik,j}Q_{ij,k} \\
=\sum_{j=1}^3\left\{\frac{L_1}{2}{|\nabla Q_j|}^2+\frac{L_2}{2}\left(\dvr{Q_j}\right)^2+\frac{L_3}{2}\nabla Q_j\cdot \nabla Q_j^T\right\},
\end{multline}
and that the bulk Landau-de Gennes energy density is
\begin{equation}
\label{eq:LdG}
f_{LdG}(Q):=a\,\mathrm{tr}\left(Q^2\right)+\frac{2b}{3}\mathrm{tr}\left(Q^3\right)+\frac{c}{2}\left(\mathrm{tr}\left(Q^2\right)\right)^2,
\end{equation}
cf. \cite{Mottram_Newton}. Here $Q_j,\,j=1,2,3$ is the $j$-th column of the matrix $Q$ and $A\cdot B=\tr\left(B^TA\right)$ is the dot product of two matrices $A,B\in M^{3\times3}.$ Further, the coefficient $a$ is temperature-dependent and in particular is negative for sufficiently low temperatures, and $c>0$.
One readily checks that the form \eqref{eq:LdG} of this potential implies that in fact $f_{LdG}$ depends only on the eigenvalues of $Q$, and due to the trace-free condition, therefore
 depends only on two eigenvalues. Equivalently, one can view $f_{LdG}$ as a function of the two degrees of orientation $S_1$ and $S_2$ appearing in \eqref{biaxial}. Furthermore,
  its form guarantees that the isotropic state $Q\equiv 0$ (or equivalently $S_1=S_2=0$) yields a global minimum at high temperatures while a uniaxial state of the form \eqref{uniaxial} where either $S_1=0,\;S_2=0$ or $S_1=S_2$ gives the minimum when temperature (i.e. the parameter $a$) is reduced below a certain critical value, cf. \cite{apala_zarnescu_01,Mottram_Newton}. In this paper we fix the temperature to be low enough so that the minimizers of $f_{LdG}$ are uniaxial. We also remark for future use that $f_{LdG}$ is bounded from below and can be made nonnegative by adding an appropriate constant. In light of this, we will henceforth assume a minimum value of zero for $f_{LdG}.$

We now turn to the behavior of the nematic on the boundary of the sample. Here two alternatives are possible. First, the Dirichlet boundary conditions on $Q$ are referred to as strong anchoring conditions in the physics literature: they impose specific preferred orientations on nematic molecules on surfaces bounding the liquid crystal. In the sequel we impose these conditions on the lateral part of the film $\partial\M\times(-h,h)$ whenever $\M$ is not closed. An alternative is to specify the anchoring energy on the boundary of the sample; then orientations of the molecules on the boundary are determined as a part of the minimization procedure. We adopt this approach, referred to as {\em weak anchoring}, on the top and the bottom surfaces of the film. Following the discussion in Section 3 of \cite{GMS}, we assume that, up to an additive constant, the anchoring energy has the form
\begin{equation}
\label{fs}
f_s(Q,\nu)=\alpha\left[(Q\nu\cdot\nu)-\beta\right]^2+\gamma{\left|\left(\mathbf{I}-\nu\otimes\nu\right)Q\nu\right|}^2,
\end{equation}
for any $\nu\in\mathbb{S}^1$ and $Q\in\mathcal A$, where $\alpha, \gamma>0$, $\beta\in\mathbb{R}$, and
\begin{equation}
\label{eq:cala}
\mathcal A:=\left\{Q\in M^{3\times3}_{sym}:\mathrm{tr}\,{Q}=0\right\}.
\end{equation}
This form of the anchoring energy requires that a minimizer of $f_s$ has $\nu$ as an eigenvector with corresponding eigenvalue equal to $\beta$. From \eqref{eq:bnds} it follows that $\beta\in\left[-\frac{1}{3},\frac{2}{3}\right]$. An alternative approach would be to extend the anchoring energy by including quartic terms \cite{fournier2005modeling} and even surface derivative terms \cite{Longa}.

Putting the three energy densities $f_e$, $f_{LdG}$ and $f_s$ together, cf. \eqref{elastic}, \eqref{eq:LdG} and \eqref{fs}, we arrive at a Landau-de Gennes type model to be analyzed in this study, given by
\begin{equation}
\label{energy}
E_h[Q]:=\int_{\Omega_h}\left\{f_e(\nabla Q)+f_{LdG}(Q)\right\}\,dV+\int_{\mathcal{M}_{-h}\cup\mathcal{M}_h}f_s(Q,\nu)\,d\mathcal{H}^2(x).
\end{equation}
Here $d\mathcal{H}^2$ represents surface measure, i.e. two-dimensional Hausdorff measure.

Again, we will consider surfaces $\mathcal{M}$ that are either closed or have a smooth boundary. In the case when the boundary is nonempty, we set $\Omega^{\mathrm{lat}}_{h}:=\partial\Omega_h\backslash\{\mathcal{M}_{-h}\cup\mathcal{M}_h\}$ and for given uniaxial data $g\in H^{1/2}(\Omega^{\mathrm{lat}}_{h};\mathcal{A})$ we prescribe the lateral boundary condition of the form
\begin{equation}
  \label{eq:bd}
  Q(X)=g(x)\ \mathrm{for}\ X\in\Omega^{\mathrm{lat}}_{h}.
\end{equation}
Note that we assume that the boundary data $g$ does not vary in the direction normal to the surface $\mathcal{M}$. {Some additional conditions on $g$ will be imposed later on in the text, cf. \eqref{Hg}.}

The admissible class of tensor-valued functions is then $Q$ lying in the Sobolev space $H^1\left(\Omega_h;\mathcal{A}\right)$ with $Q|_{\Omega^{\mathrm{lat}}_{h}}=g$, where $\mathcal{A}$ is the set of three-by-three symmetric traceless matrices defined in \eqref{eq:cala}. Throughout this work we assume that $g$ is uniaxial and is taken so that this set of admissible tensors is nonempty.

\subsection{Non-dimensionalization}
\label{secnd}

We non-dimensionalize the problem by scaling the spatial coordinates
\[\tilde{X}=\frac{X}{D},\;\tilde{x}=\frac{x}{D}\ \]
where $D:=\mathrm{diam}(\mathcal{M})$.  Set $M_2=\frac{L_2}{L_1}$ and $M_3=\frac{L_3}{L_1}$ and introduce the small non-dimensional parameter  $\epsilon=\frac{h}{D}$ representing the aspect ratio between the film thickness and the diameter of the closed surface. Then we define the non-dimensionalized elastic energy density and Landau-de Gennes potential by setting
\beq
\tilde{f}_e(\nabla_{\tilde{X}} Q):=\frac{D^2}{L_1}f_e(\nabla_X Q)
=\frac{1}{2}\sum_{j=1}^3\left\{|\nabla_{\tilde{X}} Q_j|^2+M_2\left(\dvr_{\tilde{X}}{Q_j}\right)^2+M_3\nabla_{\tilde{X}} Q_j\cdot \nabla_{\tilde{X}} Q_j^T\right\}\label{newelastic}
\eeq
and
\begin{equation}
\tilde{f}_{LdG}(Q):=\delta^2\frac{D^2}{L_1}f_{LdG}(Q)=2A\,\mathrm{tr}\left(Q^2\right)+\frac{4}{3}B\,\mathrm{tr}\left(Q^3\right)+{\left(\mathrm{tr}\left(Q^2\right)\right)}^2,\label{newLdG}
\end{equation}
respectively. Here the parameters $A:=\frac{a}{c},$ $B:=\frac{b}{c},$ and $\delta:=\sqrt{\frac{2L_1}{cD^2}}$ are all non-dimensional. 

Finally, turning to the surface energy we let $\tilde\alpha:=\frac{\alpha D}{L_1},\ \tilde{\gamma}:=\frac{\gamma D}{L_1}$, and setting
\[\tilde{f}_s(Q,\nu):=\frac{D}{L_1}f_s(Q,\nu),\]
we obtain an expression for the non-dimensionalized surface energy of the form
\begin{equation}
\label{eq:baren}
\tilde{f}_s(Q,\nu)=\tilde{\alpha}\left[(Q\nu\cdot\nu)-\beta\right]^2+\tilde{\gamma}{\left|\left(\mathbf{I}-\nu\otimes\nu\right)Q\nu\right|}^2.
\end{equation}
Now for convenience we drop all of the tildes and conclude that the total dimensionless energy is
\begin{equation}
\label{eq:8}
E_\e[Q]:=\frac{1}{L_1D}E_h[Q]=\int_{\Omega_\e}\left(f_e(\nabla Q)+\frac{1}{\delta^2}f_{LdG}(Q)\right)\,dV+\int_{\mathcal{M}_{-\varepsilon}\cup\mathcal{M}_\varepsilon}f_s(Q,\nu)\,dA.
\end{equation}
Here the rescaled domain, denoted by $\Omega_\e$, is given by
\[
\Omega_\e:=\{ X\in\R^3:\;X=x +\e t \nu
(x)\;\mbox{for}\; x \in \mathcal{M} ,\; t \in (-1,1)  \}\quad\mbox{for}\;\e<\e_0:=\frac{h_0}{D}
\]
and
\[
\mathcal{M}_{\pm \e}:=\left\{x\pm \e\nu(x):x\in\mathcal{M}\right\},
\]
where $\M$ now denotes the rescaled surface of diameter one.

Lastly, we divide by $\e$, letting $F_\epsilon[Q]:=\frac{1}{\e}E_\e[Q]$, so as to obtain an energy that is $O(1)$ for small $\e$. Hence,
\begin{equation}
\label{nden}
F_\e[Q]:=\frac{1}{\e}\int_{\Omega_\e}\left(f_e(\nabla Q)+\frac{1}{\delta^2}f_{LdG}(Q)\right)\,dV+\frac{1}{\e}\int_{\mathcal{M}_{-\varepsilon}\cup\mathcal{M}_\varepsilon}f_s(Q,\nu)\,d\mathcal{H}^2(x),
\end{equation}
where $f_e$, $f_{LdG}$ and $f_s$ are given by \eqref{newelastic}, \eqref{newLdG} and \eqref{eq:baren}, and now $F_\e$ is defined over the set of $Q$-tensors 
\beq
\mathcal{C}_g^\e:=\left\{Q\in H^1\left(\Omega_\e;\mathcal{A}\right):Q|_{\Omega^{\mathrm{lat}}_{\e}}=g\right\}.\label{Ceg}\eeq

\subsection{$Q$-tensors on a fixed domain; Surface gradients and divergences}

With an eye towards eventually passing to the $\e\to 0$ limit via $\Gamma$-convergence, we now find it convenient to re-express the $Q$-tensors, their gradients and their divergences in terms of tensors defined on the fixed domain $\M\times (-1,1)$ rather than $\Omega_\e$.  To this end, we first recall some basic identities for the surface gradient and surface divergence, for which a good reference is \cite{Simon}, Chapter 2. For any scalar-valued function $f$ defined on $\Omega_{\e}$  we
henceforth associate to it a function $\hat{f}=\hat f(x,t)$ defined on $\M\times (-1,1)$ via the formula
\beq
\hat{f}(x,t):=f\big(x +\e t \nu(x)\big).\label{fhat}
\eeq
Then, for points $X\in \Omega_\e$ and $x\in\M$ related via $X=x+t\e\nu(x)$ we readily compute that
\beq
\hat{f}_t(x,t)=\e\nabla_X f\cdot\nu(x)\label{ft}
\eeq
and for $\tau=\tau(x)$ any unit tangent vector to $\M$ at $x$ we can calculate the directional derivative as
\[
\partial_{\tau}\hat{f}=\nabla _Xf\cdot \big(\tau+\e t \partial_{\tau}\nu\big).\]
 Hence, denoting by $\{\tau_1,\tau_2\}$ an orthonormal basis for the local tangent plane to $\M$
 at $x$ and invoking \eqref{ft} and the summation convention on repeated indices we find that the surface gradient
$\nabla_{\M}\hat{f}$ is given by
\begin{eqnarray}
\label{eq:sugr}
\nabla_{\M}\hat{f}:=\partial_{\tau_j}\hat{f}\,\tau_j&&=
\big(\nabla_Xf\cdot\tau_j\big)\tau_j+\e t\big(\nabla_Xf\partial_{\tau_j}\nu\big)\tau_j\nonumber\\
&& =\nabla_Xf-\big(\nabla_Xf\cdot\nu\big)\nu+\e t\nabla_Xf \big(\partial_{\tau_j}\nu\,\tau_j\big)\nonumber\\
&&=\nabla_Xf-\frac{1}{\e}\hat{f}_t\nu+\e t\nabla_X f\nabla_{\M}\nu,\label{gradM}
\end{eqnarray}
where we recognize $\nabla_{\mathcal M}\nu$ as the shape operator.
Consequently, if we introduce the matrix-valued mapping $\Phi=\Phi(x,t;\e)$ via the formula
\beq
\Phi(x,t;\e):=\left(I+\e t\nabla_{\M}\nu(x)\right)^{-1},\label{Phidefn}
\eeq
then the previous calculation yields
\beq
\nabla_Xf=\left(\nabla_{\M}\hat{f}+\frac{1}{\e}\hat{f}_t\,\nu\right)\Phi.\label{gradXf}
\eeq
In the case where $f$ is vector-valued, the identities \eqref{gradM} and \eqref{gradXf} still hold
but with the vector quantity $\hat{f}_t\nu$ replaced by the matrix $\hat{f}_t\otimes \nu.$ Thus, in particular for $Q_j=$ 
the $j^{th}$ column of a $Q$-tensor, we find
\beq
\nabla_X Q_j=\left(\nabla_{\M}\hat{Q}_j+\frac{1}{\e}\hat{Q_j}_t\otimes\nu\right)\Phi.\label{gradXQj}
\eeq
Further, if we expand $\Phi$ in $\e$ as
\beq
\Phi(x,t;\e)\sim I-\e t\nabla_{\M}\nu(x)+O(\e^2),\label{Phiexp}
\eeq
and we use the properties $\nu\cdot\nabla_\M\nu=0=\nabla_\M\nu\cdot \nu$ resulting from the condition $\abs{\nu}=1$ one sees that
\beq
\nabla_XQ_j\sim \frac{1}{\e}\hat{Q_j}_t\otimes\nu+\nabla_\M\hat{Q}_j-\e t\nabla_\M \hat{Q_j}\nabla_\M\nu+O(\e^2).\label{gradexp}
\eeq

We then obtain a corresponding formula for the divergence of $Q_j$ in terms of $x$ and $t$ derivatives:
\beq
\dvr_X{Q_j}=\nabla_X Q_{ij}\cdot e_i=\left(\nabla_{\M}\hat{Q}_{ij}+\frac{1}{\e}(\hat{Q}_{ij})_t\nu\right)\Phi\cdot e_i,\label{newdiv}
\eeq
{where $\left\{e_i\right\}_{i=1}^3$ is an orthonormal basis in $\mathbb R^3$.} Defining the surface divergence of a vector field $F$ by $\dvr_\M F:=\nabla_\M F^{(i)}\cdot e_i$, we again expand $\Phi$ using \eqref{Phiexp}
to find
\begin{eqnarray}
\dvr_X{Q_j}&&\sim \frac{1}{\e}({\hat{Q}_{ij}})_t\nu\cdot e_i+\nabla_\M \hat{Q}_{ij}\cdot e_i-\e t\nabla_\M\hat{Q}_{ij}\nabla_\M\nu\,e_i+O(\e^2)\nonumber\\
&&= \frac{1}{\e}(\hat{Q}_{ij})_t\nu^i+\dvr_\M \hat{Q}_j-\e t\nabla_\M\hat{Q}_{ij}\nabla_\M\nu\,e_i+O(\e^2).\label{divexp}
\end{eqnarray}

\section{$\Gamma$-convergence to a surface energy defined on $\M$}
\label{s:conv}

In this section we pass to the limit $\e\to 0$ in the energy $F_\e$ given by \eqref{nden}. For convenience, we assume that an appropriate constant has been added to the Landau-de Gennes energy to guarantee that $F_\epsilon[Q]\geq 0$. {Here we are assuming that the elastic constants satisfy the conditions stated in Theorem \ref{t1} that ensure the coercivity of $F_\e$.} We wish to consider a range of asymptotic regimes corresponding to different magnitudes of $\alpha$ and $\gamma$ in the surface energy density given by \eqref{eq:baren}. To this end, we will assume that $\alpha=\alpha_0+\eps\alpha_1$ and $\gamma=\gamma_0+\eps\gamma_1$ for some nonnegative constants $\alpha_0, \alpha_1,\gamma_0,\gamma_1$. Then \eqref{eq:baren} can be written as
\begin{equation}
  \label{eq:se}
  f_s(Q,\nu)=f_s^{(0)}(Q,\nu)+\eps f_s^{(1)}(Q,\nu),
\end{equation}
where
\begin{equation}
\label{eq:fso}
f_s^{(0)}:=\alpha_0\left[(Q\nu\cdot\nu)-\beta\right]^2+\gamma_0{\left|\left(\mathbf{I}-\nu\otimes\nu\right)Q\nu\right|}^2,
\end{equation}
and
\begin{equation}
\label{eq:fsi}
f_s^{(1)}:=\alpha_1\left[(Q\nu\cdot\nu)-\beta\right]^2+\gamma_1{\left|\left(\mathbf{I}-\nu\otimes\nu\right)Q\nu\right|}^2.
\end{equation}
{Here, we can assume that $\alpha_0\alpha_1=\gamma_0\gamma_1=0$. Indeed, as will become evident later on, $f_s^{(0)}$ asymptotically vanishes at leading order in the thin film limit. Thus, if for instance, $\alpha_0\neq0$, the first term in $f_s^{(1)}$ is not present and we may take $\alpha_1=0$. 

Note that we would like to capture the asymptotic behavior of $F_\e$ for the range of parameter values, even when some of the material constants have magnitudes comparable to thickness. Mathematically, it then appears that these constants vary with thickness, even though this is not the case physically.}

Next we define the spaces 
\beq
\mathcal{C}_g:=\left\{Q\in H^1(\M\times (-1,1);\mathcal{A}):Q|_{\partial\M\times(-1,1)}=g\right\}\label{Cg}
\eeq
and
\beq
H_g:=\left\{Q\in \mathcal{C}_g:\; Q_t\equiv 0\;\mbox{a.e}.,\;f_s^{(0)}(Q(x),\nu(x))=0\mbox{ for a.e. }x\in\M\right\}
\label{Hg}
\eeq
for some uniaxial boundary data $g\in H^{1/2}\left(\partial\M;\mathcal{A}\right)$ such that the set $H_g$ is nonempty. In the case where $\partial\M=\emptyset$, this boundary condition is not present in these two definitions. 

Now we are ready to define our candidate for the $\Gamma$-limit of the sequence $\{F_\e\}$. We let $F_0:\mathcal{C}_g\to\R$ be given by

 \begin{equation}
   \label{eq:f0}
 F_0[Q]:=\left\{
     \begin{array}{ll}
       \int_{\mathcal M}\left\{f_{e}^0(\nabla_\M Q,\nu)+\frac{1}{\delta^2}f_{LdG}(Q)+2f_s^{(1)}(Q,\nu)\right\}\,d\mathcal{H}^2(x)& \mbox{ if }Q\in H_g, \\
       +\infty & \mbox{ otherwise, }
     \end{array}
 \right.
 \end{equation}
where, recalling \eqref{newelastic} we define
\begin{multline}
\label{eq:gosharius}
f_{e}^0(\nabla_\M Q,\nu):=\min_{ G\in\mathcal A}f_e( G\otimes\nu+\nabla_\M Q)\\=\frac{1}{2}\sum_{i=1}^3\left\{{\left|\nabla_{\mathcal M}{Q_i}\right|}^2+M_2{\left(\mathrm{div}_{\mathcal M}Q_i\right)}^2+{M_3}\left(\nabla_{\mathcal M}Q_i\cdot{\left(\nabla_{\mathcal M}Q_i\right)}^T\right)\right\} \\ +\min_{ G\in\mathcal A}\left[\sum_{i=1}^3\left\{\left(M_2\left(\mathrm{div}_{\mathcal M}Q_i\right)\nu+M_3{\left(\nabla_{\mathcal M}Q_i\right)}^T\nu\right)\cdot  G_i+\frac{1}{2}{\left| G_i\right|}^2+\frac{1}{2}(M_2+M_3){( G_i\cdot\nu)}^2\right\}\right].
\end{multline}
Here, as will become apparent later on, $ G$ arises as a remnant of the normal component of the gradient of $Q$.

\begin{remark}
\label{r:planar}
We wish to point out an omission in Theorem 5.1 of \cite{GMS} where the $\Gamma$-limit should have also been defined using \eqref{eq:gosharius}. When $M_2=M_3=0$ this theorem is true as stated. Otherwise, the statement and the proof should be modified as in this paper. We note that the parameter studies in Section 6 of \cite{GMS} are unaffected as they are conducted in the regime $M_2=M_3=0$.
\end{remark}


In order to phrase our $\Gamma$-convergence result we must deal with the issue that $F_\e$ and $F_0$ are defined on very different spaces, a situation common to dimension-reduction analyses involving $\Gamma$-convergence. To address this, we recall the association introduced earlier 
between any mapping, say $f$, defined on $\Omega_\e$ and the mapping $\hat{f}$ defined on $\M\times (-1,1)$, cf. \eqref{fhat}. Then we will
define the topology of the $\Gamma$-convergence as weak $H^1$-convergence in the following sense:
\beq
\mbox{We write}\quad Q_\e \stackrel{\wedge}{\rightharpoonup} Q\quad \mbox{if}\quad \hat{Q}_\e\rightharpoonup Q\;\mbox{weakly in}\; H^1(\M\times(-1,1);\mathcal{A})\label{weaktop}
\eeq
for any sequence $\{Q_\e\}\subset \mathcal{C}_g^\e$ (cf. \eqref{Ceg}) and any limit $Q\in \mathcal{C}_g$

We now state our main theorem on dimension reduction via $\Gamma$-convergence.  For those unfamiliar with the notion, we refer, for example, to \cite{DalMaso}. 

\begin{theorem}
\label{t1}
Fix $g\in H^{1/2}\left(\partial\M;\mathcal{A}\right)$ such that the set $H_g$ is nonempty. Assume that $-1< M_3<2$, and $-\frac{3}{5}-\frac{1}{10}M_3< M_2$. Let $F_\e$ be given by \eqref{nden}, with $f_e$, $f_{LdG}$ and $f_s$ given by \eqref{newelastic}, \eqref{newLdG} and \eqref{eq:se} respectively. Then $\Gamma$-$\lim_\eps{F_\eps}=F_0$ in the weak  $H^1$  topology defined in \eqref{weaktop}.
Furthermore, if a sequence $\left\{Q_\eps\right\}_{\eps>0}\subset \mathcal{C}_g^\e$ satisfies a uniform energy bound $F_\e[Q_\e]<C_0$ then there is a subsequence $\{\hat{Q}_{\e_j}\}$ such that $\hat{Q}_{\e_j}\stackrel{\wedge}{\rightharpoonup} Q$ as $\e_j\to 0$
for some $Q\in H_g^1.$
\end{theorem}

\begin{proof} Given any $Q\in \mathcal{C}_g$ we recall the association between $Q$ and $\hat{Q}\in H^1(\M\times (0,1))$ that is given by \eqref{fhat}. Using \eqref{gradexp}, \eqref{divexp} and the easily checked properties
\[
dV=dV(X)\sim \e\big(1+O(\e)\big)d\mathcal{H}^2_\M(x)\,dt,\qquad d\mathcal{H}^2_{\M_\e}\sim \big(1+O(\e)\big)d\mathcal{H}^2_\M\] we find that the energy $F_\e(Q)$ can be written in terms of $\hat{Q}$ as follows:
\begin{eqnarray}
&&F_\e[Q]\sim\nonumber\\
&&\frac{1}{2}\sum_{j=1}^3\int_{-1}^{1}\int_\M\left\{\abs{ \frac{1}{\e}\hat{Q}_{j_t}\otimes\nu+\nabla_\M\hat{Q}_j-\e t\nabla_\M \hat{Q_j}\nabla_\M\nu}^2\right.\nonumber\\
&&+M_2\left(\frac{1}{\e}(\hat{Q}_{ij})_t\nu^i+\dvr_\M \hat{Q}_j-\e t\nabla_\M\hat{Q}_{ij}\nabla_\M\nu\,e_i\right)^2\nonumber\\
&&\left.+M_3\left(  \frac{1}{\e}\hat{Q_j}_t\otimes\nu+\nabla_\M\hat{Q}_j-\e t\nabla_\M \hat{Q_j}\nabla_\M\nu\right)\cdot
\left(\frac{1}{\e}\hat{Q_j}_t\otimes\nu+\nabla_\M\hat{Q}_j-\e t\nabla_\M \hat{Q_j}\nabla_\M\nu    \right)^T\right\}\nonumber\\
&&\big(1+O(\e)\big)d\mathcal{H}^2(x)\,dt\nonumber\\
&&+\nonumber\\
&&\frac{1}{\delta^2}\int_{-1}^{1}\int_\M\left\{  2A\,\mathrm{tr}\left(\hat{Q}^2\right)+\frac{4}{3}B\,\mathrm{tr}\left(\hat{Q}^3\right)+{\left(\mathrm{tr}\left(\hat{Q}^2\right)\right)}^2\right\}\big(1+O(\e)\big)
d\mathcal{H}^2(x)\,dt\nonumber\\
&&+\nonumber\\
&&\frac{1}{\e}\int_{\M}\left\{f_s^{(0)}(\hat{Q}(x,1),\nu)+\e f_s^{(1)}(\hat{Q}(x,1),\nu)+f_s^{(0)}(\hat{Q}(x,-1),\nu)+\e f_s^{(1)}(\hat{Q}(x,-1),\nu)\right\}\nonumber\\
&&\big(1+O(\e)\big)d\mathcal{H}^2(x).
\label{ugly}
\end{eqnarray}

First, we demonstrate how one can choose a recovery sequence. If $Q_0\in \mathcal{C}_g\setminus H_g$, so that either $(Q_0)_t\not\equiv 0$ or else $f_s^{(0)}(Q_0(x),\nu(x))>0$ on a set of positive measure on $\M$, then choosing $Q_\e$ such that
$\hat{Q}_\e\equiv Q_0$, from \eqref{ugly} we readily check that
\[
\lim_{\e\to 0} F_\e[Q_\e]=+\infty=F_0[Q_0].
\] 
Notice in particular that $t$-derivatives enter at $O(\frac{1}{\e^2})$ and $f^{(0)}_s$ contributes at $O(\frac{1}{\e})$ in the energy. If, on the other hand, $Q_0\in H_g$, then of course all $t$-derivatives drop in \eqref{ugly}, as do the $1/\e$ terms from $f^{(0)}_s$ in the last integral and in the $\e\to 0$ limit, one immediately arrives at $F_0(Q_0).$ 

 Now given any $Q_0\in H_g$, we set
  \begin{equation} 
\label{eq:trump}
\hat Q_\e(x,t)=Q_0(x)+\e t\bar G(x),
\end{equation}
where $\bar G(x)$ solves \eqref{eq:gosharius} with $Q_0$ playing the role of $Q$. One technicality we must confront with this proposed recovery sequence, however, is that since $\nabla_\M \hat Q_\e=\nabla_\M Q_0+\e t\nabla_\M\bar G$ and since from \eqref{eq:gosharius} we see that $\bar G$
depends on $\nabla_\M Q_0$ {(cf. \eqref{eq:eq})}, we are in the position of taking second derivatives of the $H^1$ tensor $Q_0$. Let us first establish the success of this recovery sequence under the additional assumption that $Q_0$ is smoother than $H^1$, say $H^2$, and then we will treat the more general case at the end of the argument through a mollification procedure.

Note also that when the surface $\M$ is not closed, this proposed recovery sequence must be modified via multiplication by a cutoff function so as to maintain $g$-valued boundary data. As long as the width of the boundary layer associated with the cutoff function is of order lower than $\e$, say $\sqrt{\e}$, the contribution to the energy of that layer will be negligible. We will leave out the details of this alteration.

Now consider the expression \eqref{ugly}, evaluated using \eqref{eq:trump} as the proposed recovery sequence. Clearly, the Landau-de Gennes contribution to $F_\e$ trivially converges to its limiting value and we only need to establish convergence of elastic and surface contributions. Next we observe that the elastic energy along the recovery sequence approaches
\begin{multline}
\label{eq:carson}
\frac{1}{2}\sum_{j=1}^3\int_\M\left\{\abs{\bar G_j\otimes\nu+\nabla_\M\left(Q_0\right)_j}^2+M_2\left(\bar G_j\cdot\nu+\dvr_\M \left(Q_0\right)_j\right)^2 \right. \\
\left.+M_3\left(\bar G_j\otimes\nu+\nabla_\M\left(Q_0\right)_j\right)\cdot
\left(\bar G_j\otimes\nu+\nabla_\M\left(Q_0\right)_j\right)^T\right\}d\mathcal{H}^2(x),
\end{multline}
when $\e\to0$. In light of \eqref{eq:gosharius}, we conclude that \eqref{eq:carson} is exactly the integral over $\M$ of $f_e^0(\nabla_\M Q_0,\nu)$. 

Turning our attention to the surface energy term, we have from \eqref{ugly} that the energy contribution due to $f_s^{(0)}$ is given by
\begin{multline}
\label{eq:cruz}
\frac{1}{\e}\int_{\M}\left\{f_s^{(0)}(\hat{Q}_\e(x,1),\nu(x))+f_s^{(0)}(\hat{Q}_\e(x,-1),-\nu(x))\right\}d\mathcal{H}^2(x) \\ =\frac{1}{\e}\int_{\M}\left\{f_s^{(0)}(Q_0(x)+\e \bar G(x),\nu(x))+f_s^{(0)}(Q_0(x)-\e \bar G(x),-\nu(x))\right\}d\mathcal{H}^2(x).
\end{multline}
Since 
\begin{multline*}
f_s^{(0)}(Q_0(x)+\e \bar G(x),\nu(x))+f_s^{(0)}(Q_0(x)-\e \bar G(x),-\nu(x))\\=2\alpha_0\e^2\left[\bar G(x)\nu(x)\cdot\nu(x)\right]^2+2\gamma_0\e^2{\left|\left(\mathbf{I}-\nu(x)\otimes\nu(x)\right)\bar G(x)\nu(x)\right|}^2
\end{multline*}
on $\M$, the integral in \eqref{eq:cruz} approaches  zero. Thus the limiting contribution to the surface energy is simply
\[2\int_{\M} f_s^{(1)}(Q_0(x),\nu)\,d\mathcal{H}^2(x).\]
We conclude that the energy of the recovery sequence given by \eqref{eq:trump} approaches the $\Gamma$-limit $F_0[Q_0]$.

It remains for us to construct a recovery sequence in the general case where $Q_0$ is in $H^1$ but no smoother. An obvious approach is to mollify $Q_0$ but this mollification must be done with some care. Recall that in addition to satisfying the boundary data $g$, the tensor $Q_0$ is required to satisfy the condition 
$f_s^{(0)}(Q_0(x))\equiv0$ on $\M$. Simply convolving $Q_0$ with a standard mollifier will clearly violate both of these requirements. Maintaining the boundary condition can be handled simply enough through the straight-forward use of a smooth interpolation in a boundary layer, just as we described above for adjusting the tensor {$\bar G$} near the boundary. However, obtaining a smoother version of $Q_0$ that still gives zero contribution to the leading order surface density $f_s^{(0)}$ is not as immediate. Recall, for example, that if the constants $\alpha_0$ and $\gamma_0$ in the definition of $f_s^{(0)}$ are both positive then admissible tensors $Q_0$ must maintain the normal vector $\nu(x)$ to $\M$ as an eigenvector with corresponding eigenvalue $\beta$ at each $x$ on the curved surface $\M$.

To this end, we partition $\M$ into finitely many smooth pieces, so that say, $\M=\M_1\cup\ldots\cup {\M_n}$ and on each piece we introduce
a smoothly varying orthonormal frame   $\left\{\T,\N,\nu\right\}$ where $\left\{\T,\,\N\right\}$ is an orthonormal frame in a plane tangent to $\M$ at a given point. Such a smooth frame will not exist globally on $\M$ if for instance $\M$ is a topological sphere, hence the need for the decomposition. Then in the case $\alpha_0\gamma_0\not=0$, for example, we can introduce the scalar quantities $p_1^{(j)}$ and $p_2^{(j)}$ on each $\M_j$ by expressing $Q_0$ as
\begin{equation}
  \label{QT}
  Q_0=p_1^{(j)}(\T\otimes\T-\N\otimes\N)+p_2^{(j)}(\T\otimes\N+\N\otimes\T)+\frac{3\beta}{2}\left(\nu\otimes\nu-\frac{1}{3}I\right)
\end{equation}
so that relative to this orthonormal basis one has
a representation of $Q_0$ on $\M_j$ given by
\begin{equation}
\label{pvble}
  Q_0(x)=\left(
    \begin{array}{ccc}
      p_1^{(j)}(x)-\frac{\beta}{2} & p_2^{(j)}(x) & 0 \\
      p_2^{(j)}(x) & -p_1^{(j)}(x)-\frac{\beta}{2} & 0 \\
      0 & 0 & \beta
    \end{array}
\right).
\end{equation}
This is a change of variables invoked, for example, in \cite{bauman_phillips_park}, motivated by simulations in \cite{PhysRevLett.59.2582}. In this way, the tensor $Q_0$ is characterized by just $p_1^{(j)}$ and $p_2^{(j)}$ and by mollifying these two quantities {on each patch} we obtain a smooth approximation to the original $Q_0$ on that patch, maintaining the desired conditions that $\nu$ is always an eigenvector with corresponding eigenvalue $\beta$. {Now using the partition of unity to glue together the smooth approximation on individual patches and employing the fact that all of these approximations have the common eigenpair $\nu, \beta$, we arrive at a global smooth approximation of $Q_0$ in $H_g$.} 

If, to describe another possibility, one is working in the case where $\alpha_0=0$ but $\gamma_0>0$ so that $\nu$ must be an eigenvector but the corresponding eigenvalue is free, one can again use the representation \eqref{QT}-\eqref{pvble} but the constant $\beta$ is replaced by a third scalar unknown, say $r^{(j)}(x)$. Again mollification of $p_1^{(j)}$, $p_2^{(j)}$ and $r^{(j)}$ produces a smooth approximation to $Q_0$ on each $\M_j$ that preserves the condition $f_s^{(0)}(Q_0)=0.$

Denoting the mollification of the original tensor $Q_0\in H_g$ by the smooth sequence $\{Q_{0,\delta}\}\subset H_g$ with $\delta>0$ denoting the mollification parameter, the previously presented argument goes to show that the sequence $\{Q_{\e,\delta}\}$ of tensors defined on $\Omega_\e$ characterized by
\[
\hat Q_{\e,\delta}(x,t):=Q_{0,\delta}+\e t\bar G_\delta(x)
\]
satisfies the required property of a recovery sequence, namely
\[
\lim_{\e\to 0}F_\e[Q_{\e,\delta}]=F_0[Q_{0,\delta}].
\]
Here $G_\delta$ minimizes \eqref{eq:gosharius} for $Q=Q_{0,\delta}$.
Since the proposed $\Gamma$-limit $F_0$ is clearly continuous under $H^1$-convergence and since $Q_{0,\delta}\to Q_0$ in $H^1$, we have that $F_0[Q_{0,\delta}]\to F_0[Q_0]$, and so the existence of a recovery sequence for $Q_0$ follows by a standard diagonalization argument applied to $\{Q_{\e,\delta}\}$.

For the lower semicontinuity part of $\Gamma$-convergence, consider an arbitrary sequence $\left\{Q_\eps\right\}_{\eps>0}\subset\mathcal{C}^g_\e$ such that $\hat{Q}_\e\rightharpoonup Q_0$ in $H^1(\M\times (-1,1);\mathcal{A})$ for some $Q_0\in \mathcal{C}_g.$ Clearly we may assume
\[
\liminf_{\e\to 0} F_\e[Q_\e]<+\infty
\]
and so from \eqref{ugly}, {collecting the leading order $\mathcal{O}\left(\eps^{-2}\right)$ terms}, it is apparent that necessarily 
\begin{equation}
\label{eq:peter}
\left\|(\hat{Q}_\e)_t\right\|_{L^2}\leq C\e.
\end{equation}
Thus, $Q_0=Q_0(x)$ only. Similarly, from the strong convergence of traces under
weak $H^1$-convergence, the last integral in \eqref{ugly} will only stay finite in the $\e\to 0$ limit if $f_s^{(0)}(Q_0(x),\nu(x))=0$ as well. Hence, we may assume that $Q_0\in H_g$. Furthermore, by \eqref{eq:peter} we also have that up to a subsequence, $\frac{1}{\e}(\hat{Q}_\e)_t\rightharpoonup \bar q$ as $\e\to0$ for some $\bar q$ in $L^2(\mathcal M\times(-1,1);\mathcal A)$. It follows from {the assumption $\hat{Q}_\e\rightharpoonup Q_0$ in $H^1(\M\times (-1,1);\mathcal{A})$} that 
\[\frac{1}{\e}(\hat{Q}_\e)_t\otimes\nu+\nabla_\M\hat{Q}_\e\rightharpoonup\bar q\otimes\nu+\nabla_\M{\hat Q}_{0}\quad\text{weakly in }L^2.\]
It has been established in (\cite{Gartland_Davis}, Lemma 4.2) and \cite{Longa} that when the elastic constants satisfy the conditions $-1< M_3<2$, and $-\frac{3}{5}-\frac{1}{10}M_3< M_2$, then the elastic energy density $f_e$ is convex and consequently weakly lower semicontinuous in $H^1(\mathcal M\times(-1,1);\mathcal A)$. Hence, using \eqref{eq:gosharius} and \eqref{ugly}, we obtain
\begin{multline*}
\liminf_{\e\to0}\frac{1}{\e}\int_{\Omega_\e}f_e\left(\nabla_XQ_\e\right)dX=\liminf_{\e\to0}\int_{\M\times(-1,1)}f_e\left(\frac{1}{\e}(\hat{Q}_\e)_t\otimes\nu+\nabla_\M\hat{Q}_\e\right)d{\mathcal H}^2(x)dt\\ \geq\int_{\M\times(-1,1)}f_e\left(\bar q\otimes\nu+\nabla_\M{\hat Q}_{0}\right)d{\mathcal H}^2(x)dt\geq\int_{\M}f_e^0\left(\nabla_\M{\hat Q}_{0},\nu\right)d{\mathcal H}^2(x).
\end{multline*}
In addition to convexity, it is also shown in \cite{Gartland_Davis} that under these assumptions on the elastic coefficients one has
\begin{equation}
\label{eq:coer}
f_e(\nabla Q)\geq C{|\nabla Q|}^2
\end{equation}
pointwise for all admissible $Q$, where $C>0$ does not depend on $\M$ or $\epsilon$. Thus using Sobolev embedding and convergence of traces to handle the limits of the second and
third integrals in \eqref{ugly}, one finds
\[\liminf_{\eps\to0}F_\eps[Q_\eps]\geq F_0[Q_0].\]
This proves the second part of $\Gamma$-convergence.

Note that when $M_2=M_3=0$, the quadratic form arising in the definition of the elastic energy density is diagonal. This significantly simplifies the proof of $\Gamma$-convergence. In this case, one can always choose a trivial recovery sequence and for the lower semicontinuity the minimizer $\bar G$ vanishes. 

Finally, since the uniform energy bound implies a uniform $H^1$-bound with an $L^2$-bound on $t$-derivatives that is of order $\e$, there exists a subsequence $\{\hat{Q}_{\e_j}\}$ weakly convergent in $H^1(\M\times(-1,1);\mathcal{A})$ to a limit $Q_0$ that is independent of $t$. Further, strong convergence of traces in $L^2$ implies through the boundedness of the third integral in \eqref{ugly} that $Q_0\in H_g$.
\end{proof}

{\begin{remark}
{\em The "remnant" terms of the limiting elastic energy density $f_e^0(\nabla_\M Q_0,\nu)$ constituting the last line of \eqref{eq:gosharius} are an indication that for thin elastic shells the behavior of the minimizer in the direction normal to the surface of the film is slaved to variations along the manifold $\M$. Recall that a standard implication of $\Gamma$-convergence along with compactness is that if $\{Q_\e\}$ is a sequence of minimizers to $F_\e$ then there exists a subsequence $\{Q_{\e_j}\}$ such that $ \hat{Q}_{\e_j}\rightharpoonup Q_0$ where $Q_0$ is a minimizer of the $\Gamma$-limit $F_0.$ Consequently, to first order in $\e$, it is} not {\em the case that minimizers of $F_\e[Q]$ are obtained by trivially extending the minimizers of $F_0[Q]$ to be constant along the normals to $\M$. Indeed, with the usual association $X=x+h t\nu(x)$, we rather have that
\[Q_\e(X)\sim Q_0(x)+ht\bar{G}(x),\]
for $x\in\M$, $t\in(-1,1),$ and $h>0$ small.}

\end{remark}}

\begin{remark}
  {\em  When $M_2=M_3=0$, one can easily argue that the convergence of the subsequence is, in fact, strong. Indeed, one has 
  $F_\eps[Q]\to F_0[Q]$ for every $Q\in H_g$ viewed as an element
  of $H^1(\Omega_\e)$ that is constant along normals to $\M$. Hence $\limsup_{\e\to 0}F_\eps[Q_\eps]\leq \limsup_{\e\to 0}F_\e[Q_0]=F_0[Q_0]$. Since
\begin{multline*}
\int_{\M\times(-1,1)}f_{LdG}(\hat{Q}_\eps)\,d\mathcal{H}^2(x)\,dt+\int_{\M\times\{-1,1\}}f_s^{(1)}(\hat{Q}_\eps,\nu)\,d\mathcal{H}^2(x) \\ 
\to\int_{\M}\left(f_{LdG}(Q_0)+2f_s^{(1)}(Q_0,\nu)\right)\,d\mathcal{H}^2(x)\mbox{ as }\eps\to0,
\end{multline*}
we have
\[\limsup_{\eps\to0}\int_{\M\times(-1,1)}\left(|\nabla_\M\hat{Q}_{\eps}|^2+\frac{1}{\epsilon^2}{|(\hat{Q}_{\eps})_t|}^2\right)\,d\mathcal{H}^2(x)\,dt\leq\int_{\M}\abs{\nabla_\M Q_0}^2\,
d\mathcal{H}^2(x).\]
Combining this with the lower semicontinuity of the $L^2$-norm of the derivative due to the convexity of $f_e$, strong convergence in $\cg$ along a subsequence follows.}
\end{remark}

\section{Expression for the limiting energy $f_e^0$}
\label{s:fezero}
In the Section \ref{s:conv}, we observed that the proof of $\Gamma$-convergence is significantly simpler when $M_2=M_3=0$ because the corresponding quadratic form is diagonal and one can choose a trivial recovery sequence. In this case, the Dirichlet integral over the three-dimensional domain reduces to its analog over the manifold and the ``thin" dimension decouples from dimensions that survive in the limiting problem. The following two lemmas demonstrate that when $M_2$ or $M_3$ are present this will not be the case and there are remnants of the disappearing dimension that survive in the expression for the limiting functional. For simplicity of presentation, we will derive an explicit expression for $f_e^0$ when $M_3=0$ and then state the general formula for $f_e^0$ without proof.  The general expression can be found in the same way as in Lemma \ref{l:phantom}, albeit using significantly more cumbersome computations.
\begin{lemma}
\label{l:phantom}
Suppose that $M_3=0$ and $M_2>-\frac{3}{5}$. Then 
\begin{multline}
\label{eq:phantom}
f_e^0\left(\nabla_{\mathcal M} Q,\nu\right)=\frac{1}{2}\left\{{\left|\nabla_{\mathcal M}{Q}\right|}^2+\frac{2M_2}{M_2+2}{\left|\mathrm{div}_{\mathcal M}Q\right|}^2-\frac{M_2^2}{(M_2+2)(2M_2+3)}{\left(\nu\cdot\mathrm{div}_{\mathcal M}Q\right)}^2\right\}.
\end{multline}
\end{lemma}  
\begin{proof}
First, note that the lower bound on $M_2$ corresponds to the assumption of the Theorem \ref{t1}. When $M_3=0$, a glance at \eqref{eq:gosharius} shows that we need to minimize the function
\begin{equation}
\label{eq:phi}
\phi( G):=M_2(\nu\otimes\mathrm{div}_{\mathcal M}Q)\cdot  G+\frac{1}{2}{| G|}^2+\frac{M_2}{2}{\left| G\nu\right|}^2
\end{equation}
over the set $\mathcal A$ of symmetric matrices with the zero trace. Assuming that $ G$ is symmetric and enforcing the tracelessness of $ G$ via a Lagrange multiplier $\lambda$, we seek minimizers of 
\begin{equation}
\label{eq:phil}
\phi_\lambda( G):=M_2(\nu\otimes\mathrm{div}_{\mathcal M}Q)\cdot  G+\frac{1}{2}\left({| G|}^2+M_2{\left| G\nu\right|}^2+\lambda\tr( G)\right)
\end{equation}
among all symmetric matrices in $ G\in M^{3\times3}$, subject to the constraint $\tr( G)=0$. Thus, we need to find a pair $( \bar G,\lambda)$ that solves the problem
\begin{equation}
\label{eq:123}
\left\{\begin{array}{l}
2 \bar G+M_2(\nu\otimes\mathrm{div}_{\mathcal M}Q+\mathrm{div}_{\mathcal M}Q\otimes\nu)+M_2( \bar G\nu\otimes\nu+\nu\otimes  \bar G\nu)+\lambda I=0, \\
\tr( \bar G)=0,
\end{array}
\right.
\end{equation}
where the first equation is obtained by finding the derivative of $\phi_\lambda$ with respect to a symmetric $ G$. Taking the trace of the first equation gives
\begin{equation}
\label{eq:124}
2M_2\left(\nu\cdot\mathrm{div}_{\mathcal M}Q+ \bar G\nu\cdot\nu\right)+3\lambda=0.
\end{equation}
Multiplying the first equation respectively from the right and from the left by $\nu\otimes\nu$ and adding the results, gives
\begin{multline}
\label{eq:125}
(M_2+2)( \bar G\nu\otimes\nu+\nu\otimes  \bar G\nu)=-M_2(\nu\otimes\mathrm{div}_{\mathcal M}Q+\mathrm{div}_{\mathcal M}Q\otimes\nu) \\ -\left(2M_2\left(\nu\cdot\mathrm{div}_{\mathcal M}Q+ \bar G\nu\cdot\nu\right)+2\lambda\right)(\nu\otimes\nu).
\end{multline}
Combining \eqref{eq:124} and \eqref{eq:125} allows us to conclude that
\[ \bar G\nu\otimes\nu+\nu\otimes  \bar G\nu=\frac{\lambda}{M_2+2}(\nu\otimes\nu)-\frac{M_2}{M_2+2}(\nu\otimes\mathrm{div}_{\mathcal M}Q+\mathrm{div}_{\mathcal M}Q\otimes\nu).\]
Substituting this expression back into \eqref{eq:123} and taking trace allows us to find that
\[\lambda=-\frac{2M_2\left(\nu\cdot\mathrm{div}_{\mathcal M}Q\right)}{2M_2+3},\]
hence
\begin{equation}
\label{eq: G}
 \bar G=-\frac{M_2}{M_2+2}(\nu\otimes\mathrm{div}_{\mathcal M}Q+\mathrm{div}_{\mathcal M}Q\otimes\nu)+\frac{M_2\left(\nu\cdot\mathrm{div}_{\mathcal M}Q\right)}{(2M_2+3)}\left(\frac{M_2}{M_2+2}(\nu\otimes\nu)+I\right).
\end{equation}
Finally, evaluating \eqref{eq:phi} at this $ \bar G$ and following a sequence of trivial, but tedious calculations proves \eqref{eq:phantom}.
\end{proof}
We now give the general expression for $f^0_e$.
{\begin{lemma}
\label{l:phantom_g}
Suppose that $M_2$ and $M_3$ are defined as in Theorem \ref{t1}. Then
\begin{multline*}
f_e^0\left(\nabla_{\mathcal M} Q,\nu\right)=\frac{1}{2}{\left|\nabla_{\mathcal M}{Q}\right|}^2+\frac{M_2(M_3+2)}{2(M_2+M_3+2)}{\left|\mathrm{div}_{\mathcal M}Q\right|}^2\\+\frac{(M_3^2+2M_3-1)M_2^2+(2M_3^2+5M_3+4)M_2M_3+(M_3^2+3M_3+2)M_3^2}{2(M_2+M_3+2)(2M_2+2M_3+3)}{\left(\nu\cdot\mathrm{div}_{\mathcal M}Q\right)}^2\\+\frac{1}{2}\sum_{i=1}^3\left\{M_3\left(\nabla_{\mathcal M}Q_i\cdot{\left(\nabla_{\mathcal M}Q_i\right)}^T\right)-\frac{2M_2M_3}{M_2+M_3+2}\nu\cdot\left(\nu_i\nabla_{\mathcal M} Q_i\mathrm{div}_{\mathcal M}Q\right)\right\}\\-\frac{M_3^2}{8}{\left|\sum_{i=1}^3\nu_i\left(\nabla_{\mathcal M} Q_i+\nabla_{\mathcal M} Q_i^T\right)\right|}^2+\frac{M_3^2}{4}\frac{M_2+M_3}{M_2+M_3+2}{\left|\sum_{i=1}^3\nu_i\nabla_{\mathcal M} Q_i^T\nu\right|}^2.
\end{multline*}
\end{lemma}
The outline of the proof of Lemma \ref{l:phantom_g} is given in Appendix \ref{s:f0}.}

\section{Limiting functional when $\M$ is a surface of revolution}
\label{s:revolve}
In this section we examine the special case where $\M$ is a surface revolution. We will appeal to a description of the $\Gamma$-limit $F_0$ when the surface is presented parametrically.  The relevant formulas can be found in the appendix. 
To this end, we suppose that $\M$ is specified by the map $\Psi:\mathbb{R}^2\to\mathbb R^3$ where
\begin{equation}\label{param}
\Psi(s,\theta)=\left(
\begin{array}{c}
    a_1(s)\cos{\theta}\\
    a_1(s)\sin{\theta}\\
    a_2(s)
\end{array}
\right),
\end{equation}
with $\theta\in[0,2\pi]$ and $\rr(s):=(a_1(s),a_2(s))$ a smooth curve in $\R^2$ parametrized with respect to arclength $s\in[0,L]$ for some $L>0$. Then $\rr^\prime$ is a unit vector field that we will express in terms of an angle $\phi(s)$ via $\rr^\prime(s)=(\cos{\phi(s)},\sin{\phi(s)})$. The orthonormal frame $$\left\{\T(s,\theta),\N(s,\theta),\nu(s,\theta)\right\}$$ associated with the $(s,\theta)$-parametrization of $\M$ is
\begin{eqnarray*}
& \T(s,\theta)=\left(
\begin{array}{ccc}
    \cos{\phi(s)}\cos{\theta} \\
    \cos{\phi(s)}\sin{\theta} \\
    \sin{\phi(s)}
\end{array}
\right),\qquad
\N(s,\theta)=\left(
\begin{array}{ccc}
    -\sin{\theta} \\
     \cos{\theta}\\
     0
\end{array}
\right),  & \\
& \nu(s,\theta)=\left(
\begin{array}{ccc}
    -\sin{\phi(s)}\cos{\theta} \\
    -\sin{\phi(s)}\sin{\theta} \\
    \cos{\phi(s)}
\end{array}
\right), &
\end{eqnarray*}
so that (suppressing the variables $s$ and $\theta$) we have
\[
\Psi_{,s}=\T,\quad\Psi_{,\theta}=a_1\N,\quad\Psi_{,ss}=\T_{,s}=\phi^\prime\nu,\quad\Psi_{,s\theta}=(\cos{\phi})\N,\quad\Psi_{,\theta\theta}=(a_1\sin{\phi})\nu-(a_1\cos{\phi})\T.
\]
and we also compute that
\begin{equation}
\N_{,s}=0, \quad\nu_{,s}=-\phi^\prime \T, \quad
\T_{,\theta}=(\cos{\phi})\N, \quad \N_{,\theta}=(\sin{\phi})\nu-(\cos{\phi})\T, \quad \nu_{,\theta}=-(\sin{\phi})\N.\label{eq:sr0}
\end{equation}

Then the first and the second fundamental forms for $\M$ are given by
\begin{equation}
\mathbb I=\left(
\begin{array}{cc}
\Psi_{,s}\cdot\Psi_{,s}  & \Psi_{,s}\cdot\Psi_{,\theta}  \\
\Psi_{,s}\cdot\Psi_{,\theta}  & \Psi_{,\theta}\cdot\Psi_{,s}
\end{array}
\right)
=\left(
\begin{array}{cc}
1  &  0 \\
0  &  a_1^2
  \end{array}
\right)
\end{equation}
and
\begin{equation}
\mathbb{II}=\left(
\begin{array}{cc}
\Psi_{,ss}\cdot\nu  & \Psi_{,s\theta}\cdot\nu  \\
\Psi_{,s\theta}\cdot\nu  & \Psi_{,\theta\theta}\cdot\nu
\end{array}
\right)
=\left(
\begin{array}{cc}
\phi^\prime  &  0 \\
0  &  a_1\sin{\phi}
  \end{array}
\right),
\end{equation}
respectively. It follows that $\T$ and $\N$ correspond to principal directions with the associated principal curvatures given, up to a sign, by
\begin{equation}
\kappa_T=\phi^\prime\quad\mbox{and}\quad\kappa_N=\frac{\sin{\phi}}{a_1},\label{curv}
\end{equation}
cf. \eqref{eq:j3} and \eqref{eq:j4} in the appendix.
Further, the area element of $\M$ is given by
\[dA=\sqrt{\det{\mathbb I}}\,ds\,d\theta=a_1\,ds\,d\theta\]
and the square of the magnitude of the surface gradient of a field $u$ on $\M$ can be written as
\[{|\nabla_\M u|}^2={|u_{,s}|}^2+\frac{1}{a_1^2}{|u_{,\theta}|}^2\]
in terms of the coordinates $s$ and $\theta$.

Suppose that $M_2=M_3=\alpha_1=\gamma_1=0$ so that we are in the case of equal elastic constants and all surface energy appears at leading order. Then the tensors in the admissible class $H_g$ for the energy $F_0[Q]$ satisfy
\begin{equation}
\label{eq:sr1}
Q(s,\theta)\nu(s,\theta)=\beta\nu(s,\theta)
\end{equation}
for every $(s,\theta)\in\Omega=[0,L]\times[0,2\pi]$, where
\begin{equation}
\label{eq:sr2}
F_0[Q]=\int_\Omega\left\{{|Q_{,s}|}^2+a_1(s)^{-2}{|Q_{,\theta}|}^2+\frac{1}{\delta^2}f_{LdG}(Q)\right\}a_1(s)\,ds\,d\theta.
\end{equation}

Since the admissible $Q$ satisfy \eqref{eq:sr1}, we find it preferable from this point on to use the representation \eqref{pvble} of $Q(s,\theta)$ relative to the frame $\left\{\T(s,\theta),\N(s,\theta),\nu(s,\theta)\right\}$, so that we have 
\begin{equation}
\label{eq:sr3}
  Q(s,\theta)=\left(
    \begin{array}{ccc}
      p_1(s,\theta)-\frac{\beta}{2} & p_2(s,\theta) & 0 \\
      p_2(s,\theta) & -p_1(s,\theta)-\frac{\beta}{2} & 0 \\
      0 & 0 & \beta
    \end{array}
\right).
\end{equation}
With this stipulation,
the energy is seen to depend only on the vector ${\bf{p}}=(p_1,p_2)$ and as in \eqref{QT}, $Q$ can be expressed in the form
\begin{equation}
  \label{eq:pr}
  Q=p_1(\T\otimes\T-\N\otimes\N)+p_2(\T\otimes\N+\N\otimes\T)+\frac{3\beta}{2}\left(\nu\otimes\nu-\frac{1}{3}I\right).
\end{equation}

\begin{remark}
\label{r:pdir}
We can also choose to express $Q$ in terms of its eigenframe $({\bf n},{\bf n}^\perp,\nu)$ where ${\bf n}^\perp:=\nu\times{\bf n}$,
that is
\begin{equation}
Q=\rho\left({\bf n}\otimes{\bf n} -{\bf n}^\perp \otimes {\bf n}^\perp\right)+\frac{3\beta}{2}\left(\nu\otimes\nu-\frac{1}{3}I\right),
\end{equation} 
where ${\bf n}$ is one of the nematic directors of $Q$ and $\rho-\frac{\beta}{2}$ is its eigenvalue. If we represent ${\bf n}$ in terms of its local angle with $\T$, so that
\begin{equation}
\label{eq:nrep}
{\bf n}=\cos{\psi}\T+\sin{\psi}\N\quad\mbox{and}\quad{\bf n}^\perp=-\sin{\psi}\T+\cos{\psi}\N,
\end{equation}
then
\[Q=\rho\cos{2\psi}\,(\T\otimes\T-\N\otimes\N)+\rho\sin{2\psi}\,(\T\otimes\N+\N\otimes\T)+\frac{3\beta}{2}\left(\nu\otimes\nu-\frac{1}{3}I\right).\]
Comparing this to \eqref{eq:pr}, we conclude that 
\begin{equation}
\label{palpha}
{\bf p}=\rho(\cos{2\psi},\sin{2\psi}).
\end{equation}
 Hence, the vector ${\bf p}\in\mathbb R^2$ is related to the director ${\bf n}$ in that the angle $\mathbf p$ makes with the $x$-axis is always twice that made by ${\bf n}$ with $\T$ and the magnitude of ${\bf p}$ differs from the eigenvalue of $Q$ with respect to ${\bf n}$ by $-\beta/2$.
\end{remark}

Now we let
\begin{equation}
\label{eq:sr3.5}
\begin{aligned}
Q_1&=\T\otimes\T-\N\otimes\N, & Q_2&=\T\otimes\N+\N\otimes\T,\\ Q_3&=\nu\otimes\N+\N\otimes\nu, & Q_4&=\nu\otimes\T+\T\otimes\nu.
\end{aligned}
\end{equation}
We observe that
\begin{equation}
\label{eq:sr4}
Q_i\cdot Q_j=\tr \left(Q_j^TQ_i\right)=2\delta_{ij},
\end{equation}
for $i,j=1,\ldots,4$ { with the understanding that from now on we abandon the convention that, for tensors, subscripts refer to their columns.} Using \eqref{eq:sr0}, we find
\begin{equation}
\label{eq:sr5}
\begin{aligned}
Q_{1,s}&=\phi^\prime Q_4, & Q_{1,\theta}&=2(\cos{\phi})Q_2-(\sin{\phi})Q_3, \\ Q_{2,s}&=\phi^\prime Q_3, & Q_{2,\theta}&=(\sin{\phi})Q_4-2(\cos{\phi})Q_1,\\
(\nu\otimes\nu)_{,s}&=-\phi^\prime Q_4,\qquad & (\nu\otimes\nu)_{,\theta}&=-(\sin{\phi})Q_3,
\end{aligned}
\end{equation}
so that from \eqref{eq:pr} we have
\begin{equation}
\label{eq:sr6}
Q_{,s}=p_{1,s}Q_1+p_{2,s}Q_2+p_2\phi^\prime Q_3+\left(p_1-\frac{3\beta}{2}\right)\phi^\prime Q_4
\end{equation}
and
\begin{equation}
\label{eq:sr7}
Q_{,\theta}=\left(p_{1,\theta}-2p_2\cos{\phi}\right)Q_1+\left(p_{2,\theta}+2p_1\cos{\phi}\right)Q_2  -\left(p_1+\frac{3\beta}{2}\right)(\sin{\phi})Q_3+p_2(\sin{\phi})Q_4.
\end{equation}
With the help of \eqref{eq:sr4} we conclude that
\begin{equation}
\label{eq:sr8}
\frac{1}{2}{\left|Q_{,s}\right|}^2={|\p_{,s}|}^2+\left({|\p|}^2-3\beta p_1\right){\left(\phi^\prime\right)}^2+\frac{9\beta^2}{4}{\left(\phi^\prime\right)}^2
\end{equation}
and
\begin{equation}
\label{eq:sr9}
\frac{1}{2}{\left|Q_{,\theta}\right|}^2={|\p_{,\theta}|}^2+4\cos{\phi}\left(p_1p_{2,\theta}-p_2p_{1,\theta}\right)+{|\p|}^2\left(4-3\sin^2\phi\right) +3\beta p_1\sin^2\phi+\frac{9\beta^2}{4}\sin^2{\phi},
\end{equation}
where $\p=(p_1,p_2)$. Therefore, neglecting terms that depend on $\M$ only that would lead to additive constants after integration, we have for the elastic energy density
\begin{multline}
\label{eq:sr10}
\frac{1}{2}{|\nabla_\M Q|}^2={|\p_{,s}|}^2+\frac{1}{a_1^2}{|\p_{,\theta}|}^2+\frac{4\cos\phi}{a_1^2}\left(p_1p_{2,\theta}-p_2p_{1,\theta}\right) \\ +\left(\frac{4}{a_1^2}-3\kappa_N^2+\kappa_T^2\right){|\p|}^2+3\beta p_1\left(\kappa_N^2-\kappa_T^2\right).
\end{multline}
It is also easy to check that the Landau-de Gennes potential $f_{LdG}$ is a function of the magnitude of ${\bf p}$ only., cf. for example \cite{bauman_phillips_park}.

To gain some insight into \eqref{eq:sr10}, let us assume for simplicity that $a$ is strictly positive so that $\M$ is a surface with boundary. Then let $\beta=-\frac{1}{3}$ in the expression above to model the case when all molecules in the nematic are parallel to the surface of the film, cf. \cite{Napoli_Vergori}. Further, suppose that the field $Q$ minimizes the Landau-de Gennes energy density $f_{LdG}$ everywhere on $\M$ so that, in particular, $|\p|=\mathrm{const}$ on $\M$. Then the next to last term in \eqref{eq:sr10} is purely geometric.  Therefore, neglecting this term that would lead to an additive constant after integration,
we have
\begin{equation}
\label{eq:sr11}
\frac{1}{2}{|\nabla_\M Q|}^2={|\p_{,s}|}^2+\frac{1}{a_1^2}{|\p_{,\theta}|}^2+\frac{4\cos\phi}{a_1^2}\left(p_1p_{2,\theta}-p_2p_{1,\theta}\right)+p_1\left(\kappa_T^2-\kappa_N^2\right).
\end{equation}

Following Remark \ref{r:pdir}, we can write $\p=\rho(\cos{2\psi},\sin{2\psi})$ with $\psi$ perhaps not single-valued and $\rho$ constant. Then this expression becomes
\begin{equation}
\frac{1}{2}{|\nabla_\M Q|}^2=4\rho^2\abs{\nabla_\M\psi}^2+\frac{8\rho^2\cos\phi}{a_1^2}\psi_{,\theta}+\rho\left(\kappa_T^2-\kappa_N^2\right)\cos{2\psi}.\label{QQM}
\end{equation}
We observe from this formula that contributions to the degree of $\p$ can come from both the first term on the right due to winding of $\p$ itself and from the second term related to the rotation of the frame $(\bf{T},\,\bf{N}, \nu)$. Further, the sign of $\kappa_T^2-\kappa_N^2$ in the last term on the right determines whether the director is oriented along $\T$ or $\N$. Similar conclusions from a more general differential geometric viewpoint can be found in \cite{Napoli_Vergori} and \cite{virga_talk}.

\section{Analysis of a nematic film on a frustum}
\label{s:cone}

We conclude with an example where the surface of revolution 
$\mathcal M$ is taken to be a truncated cone or frustum. This corresponds to   ${\bf r}(s)=(a_1(s),a_2(s))=(s\cos{\phi_0},s\sin{\phi_0})$ in \eqref{param}, where $s\in[s_0,s_0+L]$ for some positive $s_0$ and $L$. Since we are interested in highlighting effects due to curvature alone, we will not impose a Dirichlet condition $g$ as we had before and instead assume
natural boundary conditions on ${\bf p}$ on each orifice of the frustum.

Figure \ref{twocones} shows the results of a numerical simulation of solving the Euler-Lagrange system associated with \eqref{eq:sr2} on the frustum subject to homogeneous Neumann boundary conditions when $\frac{1}{\delta^2}$ is large. It reveals a dichotomy in the director behavior depending on {the complement, $\phi_0$, of the angle of the opening.} When $\phi_0$ is near $\pi/2$ and the cone is narrow, the vector field ${\bf n}$ follows the generators of the cone and carries no degree with respect to geodesic circles given by the upper and lower boundaries. On the other hand, when $\phi_0$ is near zero and the cone flattens to a nearly planar domain, the field ${\bf n}$ approaches a state which carries a nonzero degree with respect to geodesics along the upper and lower boundaries.

\begin{figure}

\centering
\includegraphics[height=3in]{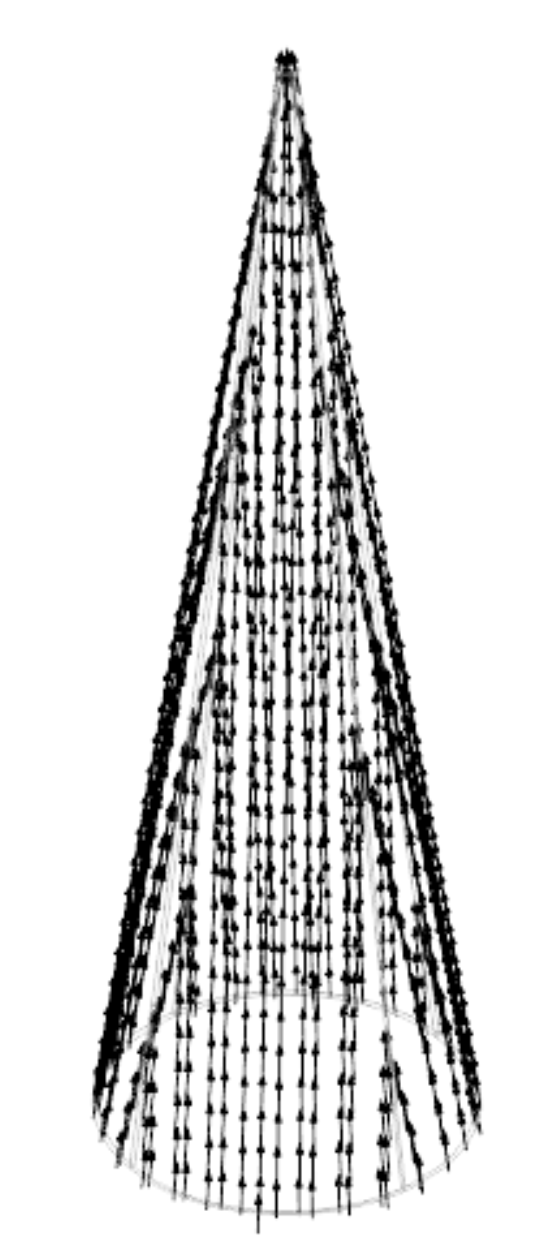}\qquad\qquad\includegraphics[height=3in]{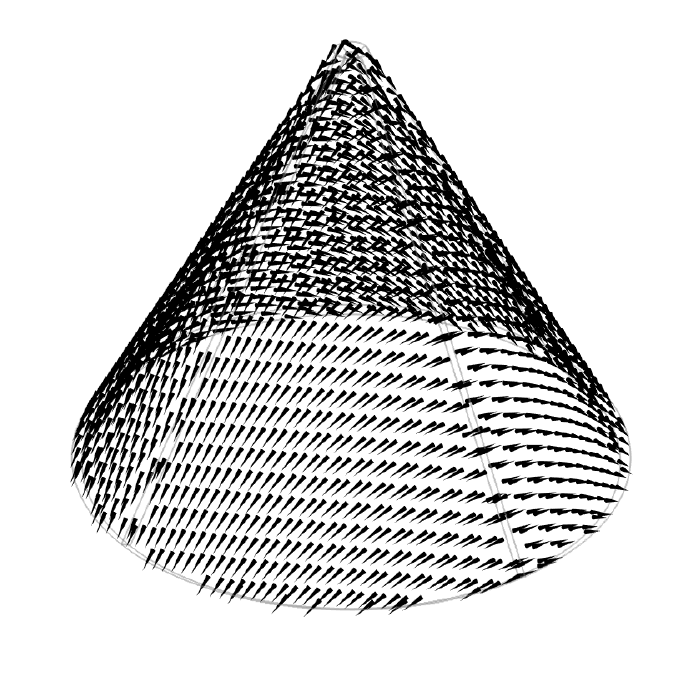}
\caption{Minimizing configurations of ${\bf n}$ for a narrow (left) and a wide (right) cones. Note that in these figures, the cones have been inverted. }
\label{twocones}
\end{figure}

 To provide some analytical basis for these numerical observations, we consider the limit $\frac{1}{\delta^2}\to \infty$ in \eqref{eq:sr2} so that we can formally assume $|{\bf p}|$ is constant so as to kill the term $f_{LdG}(Q)$; without loss of generality set $|{\bf p}|=1$. Then, referring to \eqref{palpha} we have \[{\bf p}=(\cos{2\psi(s,\theta)},\sin{2\psi(s,\theta)}).\]
 We observe from \eqref{curv} that $\kappa_T=0$ while $\kappa_N=\frac{\sin\phi_0}{a_1}$.
Thus, computing $F_0[Q]$ using \eqref{QQM} we have up to a constant that
\begin{multline*}
F_0[Q]={E_0[\psi]:=\hspace{-2mm}\int_{s_0}^{s_0+L}\hspace{-2mm}\int_0^{2\pi}\hspace{-2mm}\left[4\psi_{,s}^2+
\frac{1}{a_1^2(s)}\left(4\psi_{,\theta}^2+8\cos\phi_0\psi_{,\theta}-\sin^2{\phi_0}\cos{2\psi}\right)\right]a_1(s)d\theta\,ds}.\\
 \geq \int_{s_0}^{s_0+L}\frac{ds}{a_1(s)}\int_0^{2\pi}\left(4\psi_{,\theta}^2+8\cos\phi_0\psi_{,\theta}-\sin^2{\phi_0}\cos{2\psi}\right)d\theta \\
\geq \int_{s_0}^{s_0+L}\frac{1}{a_1(s)}\min_{\psi\in D_k}F[\psi]\,ds,
\end{multline*}
where
\[{F[\psi]:=\int_0^{2\pi}\left(4\psi_{,\theta}^2+8\cos\phi_0\psi_{,\theta}-\sin^2{\phi_0}\cos{2\psi}\right)\,d\theta}
\]
and $D_k:=\left\{\psi\in H^1([0,2\pi]):\psi(2\pi)=\psi(0)+\pi k\right\}$ for any $k\in\mathbb Z$.  Hence, the minimizer of $E_0$ for a given $k$ is independent of $s$.

Examining the expression for $E_0$ we see that any minimizer will necessarily satisfy $\psi_{,s}\equiv0,$ which leaves us to study, with a slight abuse of notation,
\begin{multline}{E_0[\psi]=\int_0^{2\pi}\left(4\psi_{,\theta}^2+8\cos\phi_0\psi_{,\theta}-\sin^2{\phi_0}\cos{2\psi}\right)\,d\theta,}\\
=8\pi \,k\cos\phi_0+\int_0^{2\pi}\left(4\psi_{,\theta}^2-\sin^2{\phi_0}\cos{2\psi}\right)\,d\theta,
\label{Dmitry}
\end{multline}
where $k$ is the winding number of the ${\bf p}$ relative to geodesic circles on the cone. Hence $k$ is twice the winding number of ${\bf n}$  along geodesic circles on the frustum, so that $\psi(2\pi)=\psi(0)+\pi k$ for some $k\in\mathbb Z$. Focusing on the last term in \eqref{Dmitry}, we observe that it corresponds to the difference in curvature squared terms in \eqref{eq:sr11}. If we only sought to optimize this term, it would force the angle $\psi$ to be zero aligning the director ${\bf n}$ with the generators $\T$ of the cone, cf. Figure \ref{twocones}(a). Setting $\psi\equiv 0$ the remaining terms in \eqref{Dmitry} would vanish, so that the total energy would be $-2\pi\sin^2{\phi_0}$. In this case ${\bf n}$ carries no degree relative to geodesic circles. 

Suppose on the other hand that $k\not= 0$, say $k=-1$. Then $\psi$ cannot be constant so this incurs some additional elastic energy given by the first term in the integrand and this positive gain competes with the negative contribution from both of the remaining terms. In Figure \ref{compare} we compare the energy of the numerically computed solution to the Euler-Lagrange O.D.E. for \eqref{Dmitry} for the case $k=-1$ to $E_0[0]$ as functions of $\phi_0$. We see that {below} some critical $\phi_0$, it is energetically preferable to have $k=-1$, in which case ${\bf n}$ does carry degree relative to geodesics. What is more, {smaller} $\phi_0$ corresponds to a gradual flattening of the frustum and convergence of minimizers to the constant state which clearly is optimal in the planar case.

Note that computationally, choosing $k$ to be any integer other than $0$ or $-1$ ends up being more expensive, cf. Figure \ref{compare}.

\begin{figure}
\centering
\includegraphics[height=2.5in]{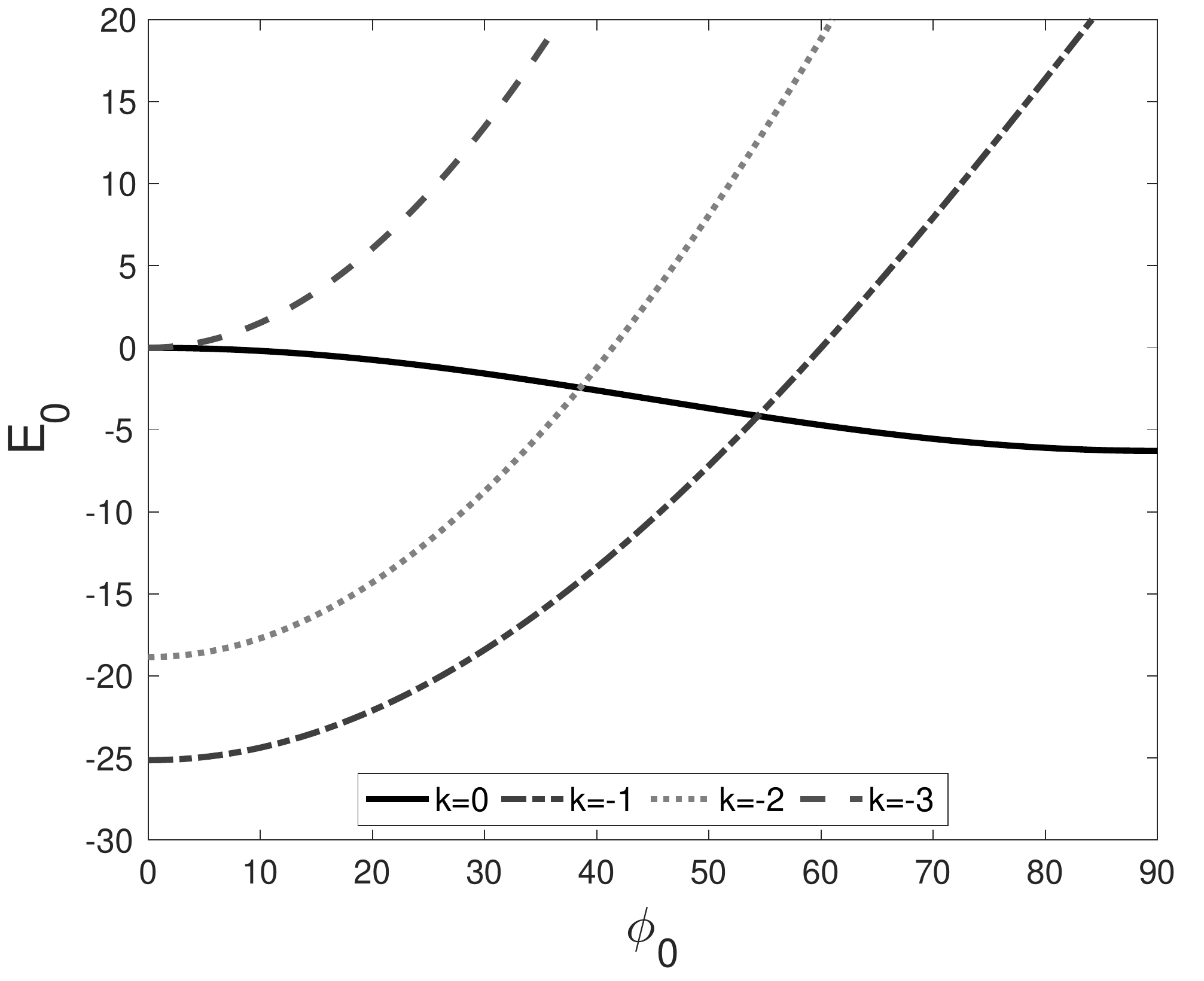}
\caption{Comparison between the energies of minimizers of $E_0$ for $k=0$, $k=-1$, $k=-2$, and $k=-3$.}
\label{compare}
\end{figure}

\section{Acknowledgements}
D.G. acknowledges support from NSF DMS-1434969. P.S. acknowledges support from NSF DMS-1101290 and NSF DMS-1362879.

\appendix

\section{Appendix: Dimension reduction for parametric surfaces}
\label{s:dimpar}
As an alternative to the approach to the dimension reduction carried out in the Section 3 here we formally outline a different argument leading to the same conclusion but using a parametric representation of the manifold $\mathcal M$. In addition to giving a different take on the limiting procedure, the parametric formulation was utilized in Sections 5 and 6.

Suppose that the geometry of the problem is as shown in Figure \ref{fig:1}. We work in non-dimensional coordinates as specified in Section \ref{secnd}. The smoothness of $\mathcal{M}$ ensures that, for a given $x_0\in\mathcal{M}$, there is an open set $U\subset\mathbb R^2$ and a smooth function $\phi:U\to\mathcal{M}$ that (a) maps $U$ homeomorphically onto an open neighborhood $V\subset\mathcal{M}$ of $x_0$ and (b) has a Jacobian matrix of rank $2$ on $U$. Since the map $\phi^{-1}:V\to U$ defines a local coordinate system on $V$, we can use the non-dimensional analog of \eqref{eq:contras} to introduce the coordinate system on $V\times\left[-\e,\e\right]$ via the smooth invertible map
\begin{equation}
\label{eq:contras1}
X=x(u)+\e t\nu(x(u)),
\end{equation}
from $U\times[-1,1]$ to $\mathbb R^3$. Note that at a given point $x(u)\in\mathcal{M}$, we have
\begin{equation}
\label{eq:j1}
X_t=\e\nu,
\end{equation}
and
\begin{equation}
\label{eq:j2}
D_uX=D_ux\left({I}+\e tA\right),
\end{equation}
where 
\begin{equation}
\label{eq:j3}
A=-\mathbb{I}^{-1}\mathbb{II},
\end{equation}
is the matrix of the shape operator and $\mathbb{I}$ and $\mathbb{II}$ are the first and second fundamental forms for $\mathcal{M}$. The shape operator $\nabla_{\mathcal{M}}\nu$ is a symmetric operator acting on the tangent space of $\mathcal{M}$ that satisfies 
\begin{equation}
\label{eq:j4}
\left(\nabla_{\mathcal{M}}\nu\right)\nu=0,\quad\left(\nabla_{\mathcal{M}}\nu\right) {\bf d}_1=\kappa_1{\bf d}_1,\quad\left(\nabla_{\mathcal{M}}\nu\right) {\bf d_2}=\kappa_2{\bf d}_2,
\end{equation}
with $\kappa_i$ and ${\bf d}_i,\ i=1,2$ being the principal curvatures and directions at $x(u)$, respectively \cite{walker2015shapes}. 

Given $X\in\Omega_\e$, let $x$ be the closest point of $\mathcal{M}$ to $X$. The gradient of a smooth vector field ${\bf a}:V\times\left[-\e,\e\right]\to\mathbb R^3$ can be decomposed into orthogonal components along and perpendicular to $\nu(x)$ by writing
\begin{equation}
\label{eq:j5}
\nabla{\bf a}=\nabla{\bf a}(\nu\otimes\nu)+\nabla{\bf a}({I}-\nu\otimes\nu).
\end{equation}
Indeed,
\begin{multline}
\label{eq:j6}
\nabla{\bf a}(\nu\otimes\nu)\cdot\nabla{\bf a}({I}-\nu\otimes\nu)=(\nu\otimes\nu)\nabla{\bf a}\cdot({I}-\nu\otimes\nu)\nabla{\bf a}\\=\tr\left\{\nabla{\bf a}({I}-\nu\otimes\nu)(\nu\otimes\nu)\nabla{\bf a}\right\}=0,
\end{multline}
so that
\begin{equation}
\label{eq:j7}
|\nabla{\bf a}|^2=\nabla{\bf a}\cdot\nabla{\bf a}={\left|\nabla{\bf a}(\nu\otimes\nu)\right|}^2+{\left|\nabla{\bf a}({I}-\nu\otimes\nu)\right|}^2.
\end{equation}
The change of variables \eqref{eq:contras1} then transforms the components of the gradient of ${\bf a}$ as follows
\begin{gather}
\nabla{\bf a}(\nu\otimes\nu)=D{\bf a}\,J^{-1}(\nu\otimes\nu)=\frac{1}{h}D{\mathbf a}\,({\mathbf e}_3\otimes\nu), \label{eq:j8}
\\ \nabla{\bf a}({I}-\nu\otimes\nu)=D{\bf a}\,J^{-1}({I}-\nu\otimes\nu)=D{\bf a}\,({I}-{\mathbf e}_3\otimes{\mathbf e}_3)J^{-1}, \label{eq:j9}
\end{gather}
where $J=\frac{\partial(X_1,X_2,X_3)}{\partial(u_1,u_2,t)}$ and $D{\mathbf a}$ is the gradient of ${\mathbf a}$ with respect to $(u_1,u_2,t)$.  Introducing the projection matrix
\begin{equation}
\label{eq:j7.5}
P_X={I}-\nu(x)\otimes\nu(x),
\end{equation}
we conclude that
\begin{gather}
\nabla{\bf a}\left({I}-P_X\right)=\frac{1}{\e}{\mathbf a}_t\otimes\nu, \label{eq:j10}
\\
\nabla{\bf a}\,P_X=D_u{\mathbf a}\,{\left({I}+\e tA\right)}^{-1}{\left(D_ux\right)}^{-1}=D_u{\mathbf a}\,\Psi(x,t;\e), \label{eq:j11}
\end{gather}
where ${\left(D_ux\right)}^{-1}$ is a left inverse of $D_ux$ and
\begin{equation}
\label{eq:psi}
\Psi(x,t;\e):={\left({I}+\e tA\right)}^{-1}{\left(D_ux\right)}^{-1}.
\end{equation}
Note that the matrix ${I}+\e tA$ is invertible when $\e$ is sufficiently small and setting $\e=0$ reduces the right hand side of \eqref{eq:j11} to
\begin{equation}
\label{eq:surfgr}
D_u{\mathbf a}{\left(D_ux\right)}^{-1}=\nabla_{\mathcal M}{\mathbf a},
\end{equation}
where $\nabla_{\mathcal M}{\mathbf a}$ is the surface gradient of ${\mathbf a}$ defined earlier in \eqref{eq:sugr}.

In non-dimensional coordinates, we can rewrite the expression for the elastic energy \eqref{newelastic} as follows
\begin{multline}
\label{elastic_bis}
f_e(\nabla Q)=\frac{1}{2}\sum_{i=1}^3\left\{{|\nabla Q_iP_X+\nabla Q_i({I}-P_X)|}^2+M_2\left(\tr{(\nabla Q_iP_X)}+\tr{(\nabla Q_i({I}-P_X))}\right)^2\right. \\ \left.+M_3(\nabla Q_iP_X+\nabla Q_i({I}-P_X))\cdot (P_X\nabla Q_i^T+({I}-P_X)\nabla Q_i^T)\right\} \\
=\frac{1}{2}\sum_{i=1}^3\left\{{\left|D_u{Q_i}\,\Psi(x,t;\e)+\frac{1}{\e}Q_{i,t}\otimes\nu\right|}^2+M_2{\left(D_u{Q_i}\cdot\Psi(x,t;\e)^T+\frac{1}{\e}Q_{i,t}\cdot\nu\right)}^2\right.\\ +\left.M_3\left(D_u{Q_i}\,\Psi(x,t;\e)+\frac{1}{\e}Q_{i,t}\otimes\nu\right)\cdot\left(\Psi(x,t;\e)^TD_u{Q_i}^T+\frac{1}{\e}\nu\otimes Q_{i,t}\right)\right\} \\
=\frac{1}{2}\sum_{i=1}^3\left\{{\left|D_u{Q_i}\,{\left(D_ux\right)}^{-1}+\frac{1}{\e}Q_{i,t}\otimes\nu\right|}^2+M_2{\left(D_u{Q_i}\cdot{\left(D_ux\right)}^{-T}+\frac{1}{\e}Q_{i,t}\cdot\nu\right)}^2\right.\\ +\left.M_3\left(D_u{Q_i}\,{\left(D_ux\right)}^{-1}+\frac{1}{\e}Q_{i,t}\otimes\nu\right)\cdot\left({\left(D_ux\right)}^{-T}D_u{Q_i}^T+\frac{1}{\e}\nu\otimes Q_{i,t}\right)\right\} +O(\e),
\end{multline}
when $\e$ is small. The same arguments that led to the proof of Threorem \ref{t1} demonstrate that the limiting elastic energy density is given by \eqref{eq:gosharius}, that is
\begin{multline}
\label{eq:gospar}
f_{e}^0\left(\nabla_{\mathcal M}Q,\nu\right)=\frac{1}{2}\min_{ G\in\mathcal A}\left[\sum_{i=1}^3\left\{{\left|\nabla_{\mathcal M}Q_i+ G_i\otimes\nu\right|}^2+M_2{\left(\mathrm{div}_{\mathcal M}Q_i+ G_i\cdot\nu\right)}^2\right.\right. \\ +\left.\left.M_3\left(\nabla_{\mathcal M}Q_i+ G_i\otimes\nu\right)\cdot\left({\left(\nabla_{\mathcal M}Q_i\right)}^T+\nu\otimes  G_i\right)\right\}\right],
\end{multline}
where $\nabla_{\mathcal M}Q_i=D_u{Q_i}\,{\left(D_ux\right)}^{-1}$ and $\mathrm{div}_{\mathcal M}Q_i=\tr\nabla_{\mathcal M}Q_i=D_u{Q_i}\cdot{\left(D_ux\right)}^{-T}$, respectively, for $i=1,\ldots,3$.

{
\section{Appendix: Outline of the proof of Lemma \ref{l:phantom_g}}
\label{s:f0}
In order find the expression for $f_{e}^0\left(\nabla_{\mathcal M}Q,\nu\right)$ recall that in \eqref{eq:gosharius} we need to minimize 
\begin{equation}
\label{eq:gosha}
\phi[G]:=\sum_{i=1}^3\left\{\left(M_2\left(\mathrm{div}_{\mathcal M}Q_i\right)\nu+M_3{\left(\nabla_{\mathcal M}Q_i\right)}^T\nu\right)\cdot  G_i+\frac{1}{2}{\left| G_i\right|}^2+\frac{1}{2}(M_2+M_3){( G_i\cdot\nu)}^2\right\}.
\end{equation}
among all $G\in\mathcal A$. To this end, set $\zeta=M_2+M_3$ and let the columns of the matrix $U\in M^{3\times3}$ be given by 
\[
U_i=M_2\left(\mathrm{div}_{\mathcal M}Q_i\right)\nu+M_3{\left(\nabla_{\mathcal M}Q_i\right)}^T\nu,
\]
where $i=1,\ldots,3$. The equation \eqref{eq:gosha} can now be written as
\begin{equation}
\label{eq:gosha1}
\phi[G]=U\cdot G+\frac{1}{2}{\left| G\right|}^2+\frac{\zeta}{2}{|G\nu|}^2.
\end{equation}
Using the same procedure as in Lemma \ref{l:phantom}, we obtain that
\begin{multline}
\label{eq:eq}
\bar G=-D(U)+\frac{\zeta}{\zeta+2}\left(\nu\otimes D(U)\nu+D(U)\nu\otimes\nu\right)\\ -\frac{\zeta\left(\zeta U\nu\cdot\nu+(\zeta+2)\,\mathrm{tr}\,U\right)}{(\zeta+2)(2\zeta+3)}\nu\otimes\nu-\frac{\zeta U\nu\cdot\nu-(\zeta+1)\,\mathrm{tr}\,U}{2\zeta+3}I
\end{multline}
minimizes \eqref{eq:gosha1}, where 
\[
D(U)=\frac{1}{2}\left(U+U^T\right).
\]
Next, substituting $\bar G$ into \eqref{eq:eq} and following a lengthy sequence of trivial calculations, the minimum value of $\phi$ is given by
\begin{multline}
\label{eq:eq1}
\phi[\bar G]=-\frac{1}{2}{|D(U)|}^2+\frac{\zeta}{\zeta+2}{|D(U)\nu|}^2-\frac{\zeta^2}{2(\zeta+2)(2\zeta+3)}{(U\nu\cdot\nu)}^2 \\ -\frac{\zeta}{2\zeta+3}(U\nu\cdot\nu)\,\tr(U)+\frac{\zeta+1}{2(2\zeta+3)}\,\tr^2(U).
\end{multline}
The conclusion of Lemma \ref{l:phantom_g} then follows from \eqref{eq:eq1} with the help of the identities 
\begin{gather*}
D(U)\nu=\frac{M_2}{2}\left(\mathrm{div}_{\mathcal M}Q+\right(\mathrm{div}_{\mathcal M}Q\cdot\nu\left)\nu\right)+\frac{M_3}{2}\sum_{i=1}^3\nu_i\nabla_{\mathcal M}Q_i^T\nu, \\
U\nu\cdot\nu=M_2\,\nu\cdot\mathrm{div}_{\mathcal M}Q, \\
\tr(U)=(M_2+M_3)\,\nu\cdot\mathrm{div}_{\mathcal M}Q.
\end{gather*}
}

\bibliographystyle{ieeetr}
\bibliography{tensor-nematic} 
\end{document}